\documentclass[a4paper,twoside,12pt]{article}
\usepackage{amssymb,amsfonts,amsmath,amsthm,latexsym}
\usepackage[pdftex,bookmarks,colorlinks=false]{hyperref}
\usepackage[hmargin=1.2in,vmargin=1.2in]{geometry}
\usepackage[auth-sc-lg,affil-sl]{authblk}
\usepackage{graphicx,caption,subcaption,centernot}
\usepackage[mathscr]{euscript}
\newtheorem{thm}{Theorem}[section]
\newtheorem{cor}[thm]{Corollary}
\newtheorem{defn}[thm]{Definition}
\newtheorem{lem}[thm]{Lemma}
\newtheorem{prob}[thm]{Problem}
\newtheorem{prop}[thm]{Proposition}

\numberwithin{equation}{section}

\def\ni{\noindent}
\def\N{\mathbb{N}}
\def\T{\mathbb{T}}
\def\G{\mathbb{G}}
\def\H{\mathbb{H}}
\def\D{\mathbb{D}}
\def\V{\mathbb{V}}
\def\R{\mathbb{R}}
\def\A{\mathbb{A}}
\def\K{\mathbb{K}}
\def\P{\mathbb{P}}
\def\W{\mathbb{W}}
\def\C{\mathbb{C}}
\def\L{\mathbb{L}}
\def\cG{\mathcal{G}}
\def\I{\mathscr{I}}
\def\sT{\mathsf{T}}

\title{\textbf{\sc  A Study on Sprout Graphs}}

\author{Johan Kok}
\affil{\small Tshwane Metropolitan Police Department\\ City of Tshwane, Republic of South Africa \\ {\tt kokkiek2@tshwane.gov.za}}

\author{Naduvath Sudev}
\affil{\small Department of Mathematics\\ Vidya Academy of Science \& Technology \\ Thalakkottukara, Thrissur - 680501, India.\\ {\tt sudevnk@gmail.com}}

\date{}

\begin{document}
\maketitle

\begin{abstract}
\ni Sprout graphs are finite directed graphs matured over a finite subset of  the non-negative time line. A simple undirected connected graph on at least two vertices is required to construct an infant graph to mature from. The maxi-max arc-weight principle and the maxi-min arc-weight principle are introduced in this paper to determine the maximum and minimum maturity weight of a sprout graph. These principles demand more mathematical debates for logical closure. Since complete graphs, paths, stars and possibly cycles form  part of the skeleton of all graphs, the introduction of results for these family of sprout graphs is expected to lay a good research foundation. 
\end{abstract}

\ni \textbf{Keywords:} Sprouting, sprout graph, infant graph, directed graph, index pattern, arc weight, maturity weight.
\vspace{0.2cm}

\noindent\textbf{AMS Classification Numbers:} 05C05, 05C20, 05C38, 05C62.

\section{Introduction}

For general notation and concepts in graph and digraph theory, we refer to \cite{BM1,CL1,GY1,DBW}. Generally all graphs mentioned in this paper other than sprout graphs, are non-trivial, simple, connected, finite and undirected graphs. The trivial graph $K_1$ will be addressed as a special case wherever applicable.

\section{Sprout Graphs}

The idea of {\em sprouting} resembles neurological growth of good or malicious networks or virus infection through Information and Communication Technology Networks. A sprout graph can generally be considered as a simple and finite directed graph matured on a time line from graphs on at least two vertices. The idea of sprouting and the notion of sprout graphs can be described as follows.

\begin{defn}\label{Defn-2.1}{\rm 
Consider a graph $G$ on $n$ vertices, where $n\ge 2$, with a fixed default vertex labeling $D=\{d_1, d_2, d_3, \ldots , d_n\} \subseteq \D = \{d_i: i \in \N\}$ and let $V(G)=\{v_i: 1 \le i \le n\}\subseteq \V = \{v_i: i \in \N\}$. Define a random bijective function $f:D \to V$ so that the vertices of $G$ are labeled according to the range of $f(D)$. The range of $f$ is called the \textit{index pattern} of $G$ and is denoted by $\I$.}
\end{defn} 

In other words, we have $\I=\{f(d_i)=v_j:1 \le i,j \le n\}$. It can be noted that a graph $G$ on $n$ vertices can have $n!$ possible index patterns. 

We denote the time line corresponding to the index pattern $\I$ by $\sT_{\I}$. By the term \textit{sprout}, we mean an ordered pair $(i,j)$ of positive integers and an edge $v_iv_j$ of $G$ can be reduced to a sprout $(i,j)$ if and only if $i<j$. If all edges of $G$ are reduced to sprouts the resultant graph is called an \textit{infant graph}. When the context is clear we shall refer to either the graph $G$ or the infant graph $G$. Invoking these definitions, the notion of a sprout graph matured from a given graph can be described as follows.

\begin{defn}\label{Defn-2.2}{\rm 
Let $\G_{\I}$ be a directed graph formed from the infant graph $G$ such that every arc $(v_i, v_j)$ of $\G_{\I}$ is formed from the sprout $(i,j)$ at time $t = |i-j|$ with, $t \in \sT_{\I} = \{0, 1,2, \ldots , m_{\I}\} \subseteq \{0, 1, 2, 3, \ldots, n-1\}$ and $m_{\I} = max|i-j|, ~ \forall~ (i,j)$. At $t= m_{\I}$ the sprouting has matured and the resultant directed graph $\G_{\I}$ is called the \textit{sprout graph} matured from the given graph $G$.}
\end{defn}

\begin{defn}{\rm 
A {\em sprouting graph}, denoted by $\G_{t=k}$, is the directed graph maturing from the given infant graph $G$ which has maturity level,  $t=k<m_{\I}$.}
\end{defn}

In real applications the sprout $(i,j)$ can evolve (or grow) over the time interval $[0,j - i)$ with arcing at $t = j- i$. Note that at $t=0$ the graph $G$ is reduced to an infant graph with sprouts $(i,j), \forall~ v_iv_j \in E(G)$ and $i<j$, attached to vertex $v_i$ and the number of sprouts attached to $v_i$ at $t=0$ is equal to $d^+_{\G}(v_i)$ in the sprout graph $\G_{\I}$. In view of this fact let us define the following notions.

\begin{defn}{\rm 
A vertex $v$ in a sprout graph $\G_{\I}$, having $d_{\G}(v_i)=d^-_{\G}(v_i)$ is called an \textit{adult vertex} and a vertex $u$ in $\G_{\I}$ having $d_{\G}(v_i)=d^+_{\G}(v_i)$ is called an \textit{initiator vertex}.}
\end{defn}

In view of the above notions, the existence of initiator and adult vertices for the sprout graphs matured from the infant graphs in respect of a given graph is established in the following proposition. 

\begin{prop}\label{prop-1.1}
All sprout graphs $\G_{\I}$, matured from the infant graphs in respect of a graph $G$ on $n \ge 2$ vertices, have at least one adult vertex and at least one initiator vertex. 
\end{prop}
\begin{proof}
Since a sprout $(n,i), i \le n-1$ can never exist, we have $d_{\G}(v_n) = d^-_{\G}(v_n)$ in $\G_{\I}$. Similarly, since a sprout $(i,1)$, $i \ge 2$ can never exist, we have $d_{\G}(v_1) = d^+_{\G}(v_1)$ in $\G_{\I}$. Therefore, every sprout graph $\G_{\I}$ has at least one initiator and an adult vertex.
\end{proof}

\ni In terms of the index patterns of two or more graphs, an index pattern for different operations of these graphs can be formed as follows.

\begin{defn}\label{Defn-2.5}{\rm 
Let the two given graphs $G_1$ and $G_2$ have the initial default index patterns $D_1=\{d_1, d_2, d_3,\ldots , d_n\}$ and $D_2=\{d'_1, d'_2, d'_3, \ldots , d'_m\}$, where $d'_j\neq d_j$, in $\D$. We define a new labeling set $D_1 \uplus D_2$ for the extended graph $G_1 \ast G_2$ by $D_1 \uplus D_2 = \{d_1, d_2, d_3, \ldots , d_n, d'_{1+n}, d'_{2+n}, d'_{3+n}, \ldots , d'_{m+n}\}$, where $\ast$ is some binary operation (either union or join of $G_1$ and $G_2$) between $G_1$ and $G_2$. }
\end{defn}

Note that the sets $D_1 \uplus D_2$ and $D_2 \uplus D_1$ need not be equal. Also, note that this notion can be applied to index patterns  $\I_1$ and $\I_2$ as well. Hence, we propose the following result.

\begin{cor}\label{Cor-1.2a}{\rm 
If $G = \bigcup\limits^{k}_{i=1}G_i$, then $\G_{\I} = \bigcup\limits^{k}_{i=1}\G_{i,\I}$ has at least $k$ adult and initiator vertices.}
\end{cor}
\begin{proof}
The result is an immediate consequence of Proposition \ref{prop-1.1}.
\end{proof}

Recall that the pendant vertices of a tree are called {\em leafs} of that tree. If the given graph $G$ has a pendant vertex, say  $v_j$, then we say that $G-v_j$ is the graph obtained by {\em lobbing off} $v_j$.  

\begin{lem}\label{Lem-2.6}
For a tree $T$ on $n$ vertices there exists at least two index patterns $\I_1$, $\I_2$ such that $\T_{\I_1}$ has exactly one adult vertex and $\T_{\I_2}$ has exactly one initiator vertex. 
\end{lem}
\begin{proof}
Consider a tree $T$ on $n$ vertices with $t$ leafs. Label the leafs randomly by $v_1, v_2, v_3, \ldots, v_t$ in an injective manner. Now, lob off the leafs to obtain the subtree $T'$ on $n-t$ vertices having $t'$ leafs. Label these leafs by $v_{t+1}, v_{t+2}, v_{t+3}, \ldots , v_{t+t'}$ injectively in a random manner. Lob off these $t'$ leafs. Repeat the procedure iteratively until we get a single vertex which can be labeled by $v_n$ or we get a $K_2$ whose end vertices can be labeled by $v_{n-1}$ and $v_n$. Then, by Definition \ref{Defn-2.2}, $v_n$ will be the unique adult vertex in the corresponding sprout graph $\T_{\I_1}$. 

In the similar way, we can find out another sprout graph $\T_{\I_2}$ whose vertices can be labeled in the reverse order so that $v_1$ is the unique initiator vertex of $\T_{\I_2}$.
\end{proof}

\begin{cor}\label{Cor-2.6a}
Every graph $G$ has at least two index patterns $\I_1$, $\I_2$ such that the corresponding sprout graph $\G_{\I_1}$ has exactly one adult vertex and $\G_{\I_2}$ has exactly one initiator vertex.
\end{cor}
\begin{proof}
The result follows immediately from Lemma \ref{Lem-2.6} and from the fact that every connected graph has a spanning tree.
\end{proof}


\begin{defn}\label{Defn-2.7}{\rm 
The \textit{arc-weight} of an arc $(v_i,v_j)$ of a sprout graph, denoted by $w(v_i,v_j)$, is defined as $w(v_i,v_j)=j-i$. If all arcs are labeled by $a_i, i = 1, 2, 3, \ldots, \epsilon(\G_{\I})$ in an injective manner, then the \textit{maturity weight} of the sprout graph $\G_{\I}$, denoted by $mw(\G_{\I})$, is defined to be $mw(\G_{\I}) = \sum \limits_{i=1}^{\epsilon (\G_{\I})}w(a_i)$. }
\end{defn}

It can be observed that the sum of arc-weights in a sprout graph $\G_{\I},$ $\forall \I$ need not be a constant, and this value depends on the random labeling of its vertices. Hence, for some index pattern $\I^{\ast}$ within the possible $n!$ index patterns we obtain, $\sum \limits_{i=1}^{\epsilon (\G_{\I^{\ast}})}w(a_i)=\min\{\sum \limits_{i=1}^{\epsilon (\G_{\I})}w(a_i)\}$. Similarly, for some index pattern $\I^{\prime}$ within the possible $n!$ index patterns, we have $\sum \limits_{i=1}^{\epsilon (\G_{\I^{\prime}})}w(a_i)=\max\{\sum \limits_{i=1}^{\epsilon (\G_{\I})}w(a_i)\}$.  Note that the index patterns $\I^{\ast}$ and $\I^{\prime}$ need not be necessarily unique. Henceforth, $\I^*_i$ and $\I^{\prime}_i$ will denote index patterns corresponding to $\min\{\sum \limits_{i=1}^{\epsilon (\G_{\I})}w(a_i)\}$ and $\max\{\sum \limits_{i=1}^{\epsilon (\G_{\I})}w(a_i)\},$ respectively.

\ni The following is a straight forward result which is important in our further studies.

\begin{lem}
Consider a graph $G$ on $n$ vertices and the index patterns $\I_1 = \{v_1, v_2, v_3,\ldots,v_n\}$ and $\I_2 = \{v_{1+k},v_{2+k},v_{3+k},\ldots,v_{n+k}\}$, where $k \in \N_0$. Then, $mw(\G_{\I_1}) = mw(\G_{\I_2})$, $\min(mw(\G_{\I^{\ast}_1}))=\min(mw(\G_{\I^{\ast}_2}))$ and\\ $\max(mw(\G_{\I^{\prime}_1}))=\max(mw(\G_{\I^{\prime}_2}))$.
\end{lem}
\begin{proof}
The results follow from the fact that $|(i+k)-(j+k)|=|i-j|$.
\end{proof}

For graphs $G_i$, $ 1\le i \le t$ and corresponding index patterns $\I_1,\I_2,\I_3, \ldots , \I_t$, consider $H = \bigcup\limits_{i=1}^{t}G_i$. Then, by Definition \ref{Defn-2.5} and Corollary \ref{Cor-1.2a} it can be followed that $mw(\H_{\I}) = \sum\limits_{i=1}^{t}mw(\G_{\I_i})$. We note that the underlying graph of $\G_{t=i}$ is a subgraph of $G$. Hence, for some $j$ and $j\le i\le m_{\I}$, the underlying graph of $\G_{t=i}$ is a spanning subgraph of $G$.

\begin{thm}\label{Thm-2.12}
For any graph $G$, there exists an index pattern $\I$ such that $\G_{t=1}$ is a directed Hamilton path of the sprout graph $\G_{\I}$ if and only if $G$ contains a Hamilton path.
\end{thm}
\begin{proof}
Assume that the given graph $G$ on $n$ vertices has a Hamilton path, say $P_n$. Label the vertices from any end vertex of $P_n$ through the consecutive adjacent vertices by $v_1, v_2, v_3, \ldots, v_n$ in an injective manner. Clearly, $(i+1)-i=1$, for all $1\le i\le n-1$. Hence, $\G_{t=1} = P^\rightarrow_n$.

Conversely, assume that for a graph $G$ on $n$ vertices, $\G_{t=1} = P^\rightarrow_n$. As time proceeds, $2 \le t \le m_{\I}$, only arcs between some vertex pairs are added. Hence, $P_n$ is contained completely in the underlying graph of the sprout graph $\G_{\I}$. Therefore, graph $G$ contains a Hamilton path.
\end{proof}

\begin{cor}
The underlying graph of the sprouting graph $\G_{t=1}$ of a given graph $G$ is acyclic.
\end{cor}
\begin{proof}
By Theorem \ref{Thm-2.12}, we have $\G_{t=1} = P^\rightarrow_n$, a directed Hamilton path and hence the underlying graph of $\G_{t=1}$ is a Hamilton path in $G$. Therefore, the underlying graph of $\G_{t=1}$ is acyclic. 
\end{proof}

\section{Two Fundamental Arc-Weight Principles}

For any three positive integers $x,y,z \in \N$ such that $x > z$ and $y > z$, we have $x-z > y -z \iff x > y$ and  $x-y \le y-z \iff x+z \le 2y$. Invoking these inequalities, we introduce two fundamental arc-weight principles. It is to be noted that the application of the principles may be a complex problem by itself. 

\subsection{The Maxi-Max Arc-Weight Principle}

The maximum maturity weight of a sprout graph $\G_{\I}$ is obtained by indexing the vertices of $G$ such that, the maximal adjacent vertex pairs exist such that the $|${\em arc-weights}$|$ are a maximum over all index patterns. 

The {\em maxi-max arc-weight principle} (MMAW-Principle) describes an index pattern $\I^{\prime}$ of the vertices of the given graph $G$ that ensures the maximum maturity weight of the sprout graph $\G_{\I^{\prime}}$. 

\subsubsection*{Fundamental MMAW-Principle Algorithm}

Consider the set of consecutive integers $I =\{1,2,3, \ldots, n\}, n \ge 2$ and let $x_{i_j} \in I$. First arrange the integers $x_{i_1}, x_{i_2}, x_{i_3}, \ldots , x_{i_n}$ such that $x_{i_k} \ne x_{i_t}$ and $|x_{i_1}- x_{i_2}| +  |x_{i_2}- x_{i_3}| + |x_{i_3}- x_{i_4}| + \ldots + |x_{i_{n-1}}- x_{i_n}|$ is a maximum. Thereafter arrange the integers such that $x_{i_k} \ne x_{i_t}$ and $|x_{i_1}- x_{i_2}| +  |x_{i_2}- x_{i_3}| + |x_{i_3}- x_{i_4}| + \ldots + |x_{i_{n-1}}- x_{i_n}|$ is a minimum.

\begin{enumerate}\itemsep0mm
\item[S-1:] Begin with: $...,n,1,n-1,..$. $\rightarrow$ 
\item[S-2:] Extend to: $...,2,n,1,n-1,3,..$. $\rightarrow$ 
\item[S-3:] Extend to: $...,n-2, 2, n, 1, n-1, 3, n-3,..$. and so on $\rightarrow$ 
\item[S-4:] Exhaust the procedure to obtain $\ell_1,\ldots n-2, 2, n, 1, n-1, 3, n-3,..., \ell_2$ with $(\ell_1, \ell_2) = (\lceil \frac{n}{2}\rceil, \lceil \frac{n}{2}\rceil+1)$ or $(\lceil \frac{n}{2}\rceil +1, \lceil \frac{n}{2}\rceil)$ or $(\lceil \frac{n-1}{2}\rceil, \lceil \frac{n-1}{2}\rceil + 1)$ $\rightarrow$
\item[S-5:] Exit.
\end{enumerate}

\subsection{The Maxi-Min Arc-Weight Principle}

The {\em maxi-min arc-weight principle} (MmAW-Principle) describes an index pattern $\I^{\ast}$ of the vertices of the given graph $G$ that ensures the minimum maturity weight of the sprout graph $\G_{\I^{\ast}}$. 

The maxi-min arc-weight principle states that the minimum maturity weight of a graph $G$ is obtained by indexing the vertices in such a way that the maximal adjacent vertex pairs exist such that the the absolute values of the arc-weights are a minimum over all index patterns.

\subsubsection*{MmAW-Principle Algorithm}

\begin{enumerate}\itemsep0mm
\item[S-1:]  Extend to: $1,2,3 \ldots, n-1, n$ $\rightarrow$
\item[S-2:] Exit.
\end{enumerate}

It is easy to see that both of these informal algorithms are well-defined and converges. These algorithms find immediate application for paths $P_n$, $n \ge 2$. 

Since a graph $G$ on $n$ vertices is a spanning subgraph of $K_n$, the vertices of $K_n$ can be labeled randomly $\I = \{v_1, v_2, v_3, ..., v_n\}$. Certainly, the graph $G$ can be obtained by removing $\frac{n(n-1)}{2}-\epsilon(G)$ carefully selected edges from $K_n$. Let $\overline{G}$ denote the complement of $G$.

It implies that if the edges are carefully selected for removal so as to ensure maxi-min arc-weights remaining in $\G_{\I^{\ast}_1}$, then $\max(mw(\overline{\G_{\I^{\ast}_1}})$, corresponding to the index pattern $\I^{\ast}_1$ concerned, is obtained. 

Similarly, if the edges are carefully selected for removal so as to ensure the maxi-max arc-weights remaining in $\G_{\I^{\prime}_2}$, then $\min(mw(\overline{\G_{\I^{\prime}_2}})$, corresponding to the index pattern $\I^{\prime}_2$ concerned, is obtained. 

From the observations above the next useful result follows.

\begin{thm}
Consider a graph $G$ on $n$ vertices. Let $\I^{\ast}_1$ and $\I^{\prime}_2$ be two index patterns such that $mw(\G_{\I^{\ast}_1})=\min(mw(\G_{\I}))$ and $mw(\G_{\I^{\prime}_2})=\max(mw(\G_{\I}))$. Then $mw(\overline{\G_{\I^{\ast}_1}})=\max(mw(\overline{\G_{\I}}))$ and $mw(\overline{\G_{\I^{\prime}_2}})=\min(mw(\overline{\G_{\I}}))$.
\end{thm}
\begin{proof}
\textbf{(i)} Let $\I^{\ast}_1$ be such that $mw(\G_{\I^{\ast}_1})=\min(mw(\G_{\I}))$. It implies that maximal number of edges with {\em maximum index differences} have been removed from $K_n$ to obtain $G$. Hence, the {\em maxi-max index differences} edges are the edges of $\overline{G}$. Therefore, $mw(\overline{\G_{\I^{\ast}_1}}) = \max(mw(\overline{\G_{\I}}))$. 

\ni \textbf{(ii)} A similar reasoning as in (i) can be applied to prove part (ii).
\end{proof}

\begin{lem}\label{Lem-3.2}
If for a graph $G$ an index pattern $\I$ exists such that $\sT_{\I} = \{0,1\}$ then, $mw(\G_{\I})=\min(mw(\G_{\I}))=\epsilon(G)$.
\end{lem}
\begin{proof}
For any index pattern $\I$ we have $|w(v_i,v_j)|\ge 1$, for any arc $(v_i,v_j)$ in $\G_{\I}$. Hence, the proof is obvious.
\end{proof}

\ni Lemma \ref{Lem-3.2} implies that $mw(\G_{\I})=\epsilon(G)$ if and only if $G=P_n$, where $n \ge 1$.

\section{Sprout Graphs of Certain Classes of Graphs}

\ni Since complete graphs, paths and possibly cycles and stars amongst others form  part of the skeleton of all graphs the introduction of sprouting to these graph classes will nourish further studies.
 
\subsection{Sprouting of Complete Graphs}

\begin{prop}\label{Prop-3.1}
For all indexing sets $\I$, the maturity weight of the sprout graph of a complete graph $K_n$ is $mw(\K_{n,{\I}})=\sum \limits_{j=1}^{n-1} \sum \limits_{i=1}^{n-j}(n-i)=\sum \limits_{j=1}^{n-1} \sum \limits_{i=1}^{n-j}i$.
\end{prop}
\begin{proof}
Randomly label the vertices of the complete graph $K_n$ by $v_1, v_2, v_3, \ldots, v_n$. Now, consider the matured sprout graph $\K_{n,{\I}}$. Regardless of the random indexing of the vertices we have the following arc-weights, $w(v_1,v_i)=i-1 ~ \forall\, 2\le i\le n$, $w(v_2,v_i)=i-2~\forall\, 3\le i\le n$, \ldots\ldots, $w(v_{n-1}, v_n)=1$. Hence, by summing all columns carrying equal arc-weights, across all the above mentioned rows, we have $mw(\K_{n,{\I}})= \sum \limits_{i=1}^{n-1}(n-i) + \sum \limits_{i=1}^{n-2}(n-i) + \sum \limits_{i=1}^{n-3}(n-i) + \ldots +  \sum \limits_{i=1}^{n-(n-1)}(n-i) = \sum \limits_{j=1}^{n-1} \sum \limits_{i=1}^{n-j}(n-i)=\sum \limits_{j=1}^{n-1} \sum \limits_{i=1}^{n-j}i$ and hence the result follows.
\end{proof}

\begin{cor}\label{Cor-4.2a}
For every index pattern $\I$, the complete sprout graph $\K_{n,{\I}}$ has one (unique) adult vertex, $v_n$  and one (unique) initiator vertex, $v_1$.
\end{cor}
\begin{proof}
Write $\K_{n,{\I}}$ as $\K_n$ for brevity. Since any vertex $v_i, i<n$ is always a tail to $v_n$, the vertex $v_i$ will always have $d^+_{\K_n}(v_i) \ge 1$ in $\K_{n,{\I}}\implies d^-_{\K_n}(v_i) < d_{\K_n}(v_i)$.  Since it contradicts Definition \ref{Defn-2.2}, the result follows from Proposition \ref{prop-1.1}. By similar arguments, we can establish the result for the unique initiator vertex also. 
\end{proof}

\begin{lem}
For a graph $G$ on $n$ vertices, we have $\max(mw(\G_{\I})) \le \min(mw(\K_{n,{\I}}))=\max(mw(\K_{n,{\I}}))$.
\end{lem}
\begin{proof}
It follows from Corollary \ref{Cor-4.2a} that $\min(mw(\K_{n,{\I}}))=\max(mw(\K_{n,{\I}}))$. Since $|\epsilon (K_n)| \ge |\epsilon(G)|$, where $G$ is a graph on $n$ vertices, the removal of edges from $K_n$ to obtain $G$ results in reducing the corresponding terms in the summation  $mw(\K_{n,{\I}})=\sum\limits_{i=1}^{n-1}(n-i) + \sum \limits_{i=1}^{n-2}(n-i) + \sum \limits_{i=1}^{n-3}(n-i) + \ldots +  \sum \limits_{i=1}^{n-(n-1)}(n-i) = \sum \limits_{j=1}^{n-1} \sum \limits_{i=1}^{n-j}(n-i) = \sum \limits_{j=1}^{n-1} \sum \limits_{i=1}^{n-j}i$, to zero. Therefore, $max(mw(\G_{\I})) \le  min(mw(\K_{n,{\I}}))=\max(mw(\K_{n,{\I}}))$.
\end{proof}

\subsection{Sprouting of Paths}

\begin{prop}\label{Prop-3.4}
For the path $P_n$, for $n\ge 2$, we have $\min(mw(\P_{n,{\I^{\ast}_1}}))=n-1$, and $\max(mw(\P_{n,{\I^{\prime}_2}})) = \sum\limits_{i=0}^{\lceil \frac {n}{2}\rceil-2}(2n-3 -4i)$.
\end{prop}
\begin{proof}
Consider the path $P_n, n \ge 2$ as a graph with its $n$ vertices seated on a horizontal line. The labeling of vertices of $G$ can be done as explained below.

\ni \textbf{(i)} Label the vertices consecutively by $v_1, v_2, v_3, ..., v_{n-1}, v_n$ , from the leftmost vertex onwards. Let this be the index pattern $\I^{\ast}_1$. Clearly, we have $m_{\I^{\ast}_1}=\max\{|i-j|\}=1$, for all sprouts $(i,j)$ and hence $\T_{\I^{\ast}_1}=\{0, 1\}$. Therefore, arcs having arc-weight $1$, will arc at $t=1$. Since there are exactly $(n-1)$ such arcs in $\P_{n,{\I^{\ast}_1}}$, by Lemma \ref{Lem-3.2}, the first part of the result follows. 

\ni \textbf{(ii)} Label the vertices from left to right consecutively with the default labelling $\{d_i: 1 \le i \le n\}$. Then we have to consider the following cases.

\ni \textit{Case 1:} Let $n$ be odd. Label the central vertex $d_{\lceil \frac{n}{2}\rceil +1}$ to be $v_1$. Label the vertices adjacent to $v_1$ respectively $v_n$ and $v_{n-1}$ in accordance with Step-2 of MMAW-Principle algorithm. Now, label the other vertices exhaustively in accordance with the MMAW-Principle algorithm to get the index pattern $\I^{\prime}_2 =\{v_{\ell_1}, \ldots , v_{n-2}, v_2, v_n,\\ v_1, v_3, v_{n-3}, \ldots , v_{\ell_2}\}$, $(\ell_1, \ell_2) = (\lceil \frac{n}{2}\rceil, \lceil \frac{n}{2}\rceil+1)$ or $(\lceil \frac{n}{2}\rceil +1, \lceil \frac{n}{2}\rceil)$ or $(\lceil \frac{n-1}{2}\rceil, \lceil \frac{n-1}{2}\rceil + 1)$.  Then, by the  MMAW-Principle and invoking Definition \ref{Defn-2.7}, we have the required condition $\max(mw(\P_{n,{\I^{\prime}_2}}))=\sum\limits_{i=0}^{\lceil \frac {n}{2}\rceil-2}(2n-3-4i)$.

\ni \textit{Case 2:} Let $n$ be even. Now, the path does not have a central vertex, instead a pair of central vertices exists. Without loss of generality, label the rightmost central vertex (that is, the $\frac{n+1}{2}$-th vertex) by $v_1$ and label the vertex to the left adjacent to $v_1$ by $v_n$ and the vertex to the right adjacent to  $v_1$ by $v_{n-1}$. Proceed with this labeling exhaustively as explained in Case-1. Therefore, the required result follows as explained in Case-1.
\end{proof}

\begin{cor}
For the path $P_n$ we have $\min(mw(\P_{2, t=1})) = \max(mw(\P_{2, t=1}))$ and $\min(mw(\P_{n,{\I^{\ast}_1}}))=mw(\P_{n, t=1}) < \max(mw(\P_{n,{\I^{\prime}_2}}))$, for $n \ge 3$.
\end{cor}
\begin{proof}
The result is an immediate consequence of Proposition \ref{Prop-3.4}.
\end{proof}

From Proposition \ref{Prop-3.1} and Proposition \ref{Prop-3.4},  we have for a graph $G$, $n-1\le mw(\G_{\I})\le \sum\limits_{j=1}^{n-1}\sum \limits_{i=1}^{n-j}i$. Hence, we get a result which states that $\forall n\in \N$, there exist a graph $G$ and an index pattern $\I$ for which $\min(mw(\G_{\I}))=n$. The graph $G$ is the path $P_{n+1}$ with index pattern found in the first part of the proof of Proposition \ref{Prop-3.4}. A similar result cannot be found for $\max(mw(\G_{\I}))$. 

\begin{thm}\label{Thm-3.6}[Zan\'e's\footnote{The first author wishes to dedicate this theorem to Zan\'e van der Merwe who it is hoped, will grow up to be a great mathematician.}]
Consider the set of graphs $\cG =\{G:\epsilon(G)=q\}$. For a graph $H \in \cG$, we have $\min(mw(\H_{\I}))=\min(\min(mw(\G_{\I'})))$ if and only if $H \cong P_{q+1}$.
\end{thm}
\begin{proof}
Clearly the result holds for $q=1,2,3$. Assume the result holds for all $G\in \cG$ with $4\le \epsilon(G)\le k$. Hence, $\min(mw(\P_{k+1,t=1}))=\min(\min(mw(\G_{\I'})))$, $\epsilon(G)=k$. Now, consider consider a graph $G$ with $\epsilon(G)=k+1$ and let $H \cong P_{k+2}$. Clearly, $\min(mw(\P_{k+2,t =1}))=\min(mw(\P_{k+1,t=1}))+1$. It is the minimum increase in maturity weight possible and hence, $\min(\min(mw(\G_{\I'})))=\min(mw(\H_{\I}))=\min(mw(\P_{k+2,t=1}))$.

Conversely, assume there exists a graph $H \ncong P_{k+2}$ such that $\min(mw(\H_{\I}))=\min(\min(mw(\G_{\I'})))$, with $\epsilon(H)=k+1$. Then, it follows that $P_{k+1}$ is a subgraph of $H$. Hence, to add the additional edge an additional pendant vertex (leaf) was added to $P_{k+1}$ to obtain $H$. This, however, implies that the increase in minimum maturity weight by $\min(mw(\H_{\I}))-\min(mw(\P_{t=1}))\ge 2$. It is a contradiction, since $\min(mw(\P_{k+2,t=1}))-\min(mw(\P_{k+1,t=1}))=1$. Therefore, we must have $H\cong P_{k+2}$.
\end{proof}

Note that we can not find a complete graph $K_q$ such that $\epsilon(K_q)=n$ for all integral values of $n$. Hence, a result analogous to Theorem \ref{Thm-3.6} to determine $\max(\max(mw(\G_{\I})))$ does not exist.

\subsection{Sprouting of Cycles}

\begin{prop}\label{Prop-4.7}
For the cycle $C_n$, where $n\ge 4$, we have $\min(mw(\C_{n,{\I^{\ast}_1}}))=2(n-1)$ and $\max(mw(\C_{n,{\I^{\prime}_2}}))=\max(mw(\P_{k+1,{\I^{\prime}_2}}))+1=\sum\limits_{i=0}^{\lceil \frac {n}{2}\rceil-2}(2n-3 -4i) + 1$. 
\end{prop}
\begin{proof}
Identify the cycle $C_n, n \ge 4$ as the graph with the $n$ vertices seated on the circumference of a circle with a vertex seated centre at the top. Then, the labeling of vertices of $C_n$ can be done as explained below.

\ni \textbf{(i)} Label the top vertex by $v_1$ and label the other vertices clockwise $v_2, v_3, ..., v_n$. Call the index pattern $\I^{\ast}_1$. Clearly, the arc-weights, $w(v_i,v_j)=1$ except for the arc $w(v_1,v_n)=n-1$. Hence, $\T_{\I^{\ast}_1} =\{0, 1, n-1\}$. Therefore, all arcs  having arc-weight $1$, will arc at $t=1$. There are exactly $(n-1)$ such arcs in $\C_{n,t=1}$ and the last arc $(v_1,v_n)$ arcs at $t=(n-1)$, so $mw(\C_{n,t= n-1})=2(n-1)$. Without loss of generality, interchange the vertex labeling $v_n$ and $v_i$, $i<n$ to obtain $\I'$. The only possible decrease in the maturity weight is on condition that $n-2<i<n$, $i \in \R$. For $i=n-1$ we have $mw(\P_{n,{\I'}})=mw(\P_{n,t=n-1})$. Hence, $\min(mw(\P_{n,{\I^{\ast}_1}}))=mw(\P_{n,t=n-1})$. 

\ni \textbf{(ii)} Consider the path $P_3$ and label the vertices $v_3, v_1, v_2$ in accordance with the MMAW-Principle. Then, here the end vertices have index difference $1$ hence\\ $\max(mw(\C_{3,t=2}))=4=\max(mw(\P_{3,t=2}))+1$, as the end vertices has the indexes $\lceil \frac{3}{2}\rceil$, $\lceil \frac{3}{2}\rceil + 1$ respectively. Next, assume that the result holds for $C_k$, $k \ge 4$. Hence, $\max(mw(\C_{k,{\I^{\prime}_2}}))=\max(mw(\P_{k,1}))+1$ and the end vertices of the path $P_k$ have indexes either $\lceil \frac{k}{2}\rceil$, $\lceil \frac{k}{2}\rceil+1$ or $\lceil \frac{k-1}{2}\rceil$, $\lceil \frac{k-1}{2}\rceil+1$ respectively. Now consider the path $P_{k+1}$. Clearly, after labeling the vertices in accordance with the MMAW-Principle, the end vertices have indexes either $\lceil \frac{k+1}{2}\rceil$, $\lceil \frac{k+1}{2}\rceil + 1$ or $\lceil \frac{k+1}{2}\rceil$, $\lceil \frac{k}{2}\rceil + 1$.  In both cases the index difference between the end vertices is $1$ and hence the result $\max(mw(\C_{k,{\I^{\prime}_2}}))=\max(mw(\P_{k+1,{\I^{\prime}_2}}))+1$ holds. Hence, the result follows by induction.
\end{proof}

\begin{cor}
For the cycle $C_n$ we have $\min(mw(\C_{3,t=1}))=\max(mw(\C_{3,t=1}^s))$ and $\min(mw(\C_{n,{\I^{\ast}_1}}))=mw(\C_{n,t=(n-1)})<\max(mw(\C_{n,{\I^{\prime}_2}}))$, where $n \ge 4$.
\end{cor}
\begin{proof}
The proof follows immediately from Proposition \ref{Prop-4.7}.
\end{proof}

From Proposition \ref{Prop-4.7}, it follows that for every positive integer $n\ge 3$, there exist a graph and an index pattern $\I$ for which $\min(mw(\G_{\I}))=2n$. The graph is the cycle $C_{n+1}$ with an index pattern found in the first part of the proof of Proposition \ref{Prop-4.7}. An analogous result cannot be found for $\max(mw(\G_{\I}))$. 

It has been established that if two different random index patterns of a graph $G$ say $\I_1$ and $\I_2$ result in $\T_{\I_1}$ and $\T_{\I_2}$ respectively, such that $\T_{\I_1}=\T_{\I_2}$ then, $\T_{\I_1}=\T_{\I_2} \centernot\implies mw(\G_{\I_1}) = mw(\G_{\I_2})$. 

\subsection{Sprouting of  Stars}

\begin{thm}\label{Thm-4.9}
The sprout graph of star $K_{(1,n)}$ has 
\begin{enumerate}\itemsep0mm
\item[(i)] $\min(mw(\K_{(1, n),{\I^{\ast}_1}}))=
\begin{cases}
2\sum \limits_{i=1}^{\lceil \frac{(n+1)}{2}\rceil-1} i + \lceil \frac{(n+1)}{2}\rceil, & \text{if $n\ge 3$ and odd},\\
2\sum \limits_{i=1}^{\lceil \frac{(n+1)}{2}\rceil-1} i, & \text{if $n\ge 2$ and even}.
\end{cases}$
\item[(ii)] $\max(mw(\K_{(1, n),{\I^{\prime}_2}}))= \sum\limits_{i=1}^{n}i,~\forall n \in \N.$
\end{enumerate}
\end{thm}
\begin{proof} 
\ni \textbf{(i)} First consider the star graph $K_{(1,3)}$. Note that in the table that follows; $i, j, k, l \in \{1,2,3,4\}$. The possible $4!$ index patterns with the corresponding values of $mw(\K_{(1,3),\I})$ are given in the following table.

\begin{tabular}{|c|c|c|c|c|}
\hline
central vertex  $v_i$ &  leaf $v_j$ &  leaf $v_k$ & leaf $v_l$ & $mw(\K_{(1,3),\I})$\\
\hline
1 & 2 & 3 & 4 & 6\\
\hline
1 & 2 & 4 & 3 & 6\\
\hline
1 & 3 & 2 & 4 & 6\\
\hline
1 & 3 & 4 & 2 & 6\\
\hline
1 & 4 & 2 & 3 & 6\\
\hline
1 & 4 & 3 & 2 & 6\\
\hline
2 & 1 & 3 & 4 & 4\\
\hline
2 & 1 & 4 & 3 & 4\\
\hline
2 & 3 & 1 & 4 & 4\\
\hline
2 & 3 & 4 & 1 & 4\\
\hline
2 & 4 & 1 & 3 & 4\\
\hline
2 & 4 & 3 & 1 & 4\\
\hline
3 & 1 & 2 & 4 & 4\\
\hline
3 & 1 & 4 & 2 & 4\\
\hline
3 & 2 & 1 & 4 & 4\\
\hline
3 & 2 & 4 & 1 & 4\\
\hline
3 & 4 & 1 & 2 & 4\\
\hline
3 & 4 & 2 & 1 & 4\\
\hline
4 & 1 & 2 & 3 & 6\\
\hline
4 & 1 & 3 & 2 & 6\\
\hline
4 & 2 & 1 & 3 & 6\\
\hline
4 & 2 & 3 & 1 & 6\\
\hline
4 & 3 & 1 & 2 & 6\\
\hline
4 & 3 & 2 & 1 & 6\\
\hline
\end{tabular} 

Clearly, $\min(mw(\K_{(1,3),t=2})) = 4 =2\sum \limits_{i=1}^{\lceil \frac{(3+1)}{2}\rceil - 1}i + \lceil \frac{(3+1)}{2}\rceil$. Therefore, the results holds for $K_{(1,3)}$. Next, assume it holds for $K_{(1,q)},~q>3$ and $q$ is odd. Hence, it is assumed that $\min(mw(\K_{(1,q),\I}))= 2\sum\limits_{i=1}^{\lceil\frac{(q+1)}{2}\rceil-1}i +\lceil\frac{(q+1)}{2}\rceil$. 

Now, consider the graph $K_{(1,q+2)}$. We have $\lceil\frac{q+3}{2}\rceil=\lceil\frac{(q+1)}{2}\rceil +1$. Hence, the central vertex index increases to $\lceil \frac{(q+1)}{2}\rceil +1$. This results in the central vertex, $v_{\lceil \frac{(q+1)}{2}\rceil}$ in $K_{(1,q)}$ to become a leaf in $K_{(1,q+2)}$, and vertex $v_{(\lceil \frac{(q+1)}{2}\rceil +1)}$ becomes the central vertex in $K_{(1,q+2)}$.  Thus, for all $v_i, i <\lceil\frac{q+1}{2}\rceil$ in $K_{(1,q)}$, the difference $|\lceil \frac{(q+1)}{2}\rceil - i|$ increases by $1$ in $K_{(1,q+2)}$. The value $|(\lceil \frac{(q+1)}{2}\rceil +1) - \lceil \frac{(q+1)}{2}\rceil|$ repeats twice due to the index interchanging of the central vertex. Then, exact mirror values follow with the value $\lceil \frac{q+1}{2}\rceil +1=\lceil\frac{q+3}{2}\rceil$, added as well. Hence, the result (i) holds for $K_{(1,q)},q>3$ and $q$ is odd. Hence, in general it follows that $\min(mw(\K_{(1,n),{\I^{\ast}_1}}))=2\sum\limits_{i=1}^{\lceil\frac{(n+1)}{2}\rceil-1}i)+\lceil \frac{(n+1)}{2}\rceil$, if $n \ge 3$ and is odd.
 
Now, consider the graph $K_{(1,2)}$. Note that in the table that follows; $i, j, k \in \{1,2,3\}$. The possible $3!$ index patterns with the corresponding values $mw(\K_{(1,2),\I})$ are given in the following table.  

\begin{tabular}{|c|c|c|c|c|}
\hline
Central vertex $v_i$ & leaf $v_j$ & leaf $v_k$ & $mw(\K_{(1,2),\I})$\\
\hline
1 & 2 & 3 & 3\\
\hline
1 & 3 & 2 & 3\\
\hline
2 & 1 & 3 & 2\\
\hline
2 & 3 & 1 & 2\\
\hline
3 & 1 & 2 & 3\\
\hline
3 & 2 & 1 & 3\\
\hline
\end{tabular} 

Clearly, $\min(mw(\K_{(1,2),\I}))=2=2\sum\limits_{i=1}^{\lceil \frac{(2+1)}{2}\rceil-1}i$. Hence, the results holds for $K_{(1,2)}$. Assume it holds for $K_{(1,q)}, q>2$ and $q$ is even. Hence, it is assumed that $\min(mw(K_{(1,q),\I}))=2\sum\limits_{i=1}^{\lceil\frac{(q+1)}{2}\rceil-1}i$. Now, consider the graph $K_{(1,q+2)}$. As explained in the first part of (i), we have that $\lceil\frac{(q+3)}{2}\rceil= \lceil \frac{(q+1)}{2}\rceil +1$. Hence, the central vertex index increases to $\lceil \frac{(q+1)}{2}\rceil +1$. This results in the central vertex $v_{\lceil \frac{(q+1)}{2}\rceil}$ in $K_{(1,q)}$ to become a leaf in $K_{(1,q+2)}$ and vertex $v_{(\lceil \frac{(q+1)}{2}\rceil +1)}$ becomes central vertex in $K_{(1,q+2)}$.  Thus, for all $v_i, i < \lceil \frac{(q+1)}{2}\rceil$ in $K_{(1,q)}$, the difference $|\lceil \frac{(q+1)}{2}\rceil - i|$ increases by $1$ in $K_{(1,q+2)}$. Therefore, as explained in the previous case, the result follows.

\ni \textbf{(ii)} By labeling the central vertex by $v_1$ and the leafs by  $v_2, v_3,\ldots, v_{n+1}$  the result can be proved similarly explained in $(i)$.
\end{proof}

\begin{cor}\label{Cor-4.10}
For the sprout star $\K_{(1,n),\I}$, the central vertex is indexed $\ell$ where
\begin{equation*} 
\ell \in
\begin{cases}
\{\lceil\frac{(n+1)}{2}\rceil, \lceil \frac{(n+1)}{2}\rceil + 1 \}, & \text{for $\min(mw(\K_{(1,n),\I}))$ if $n\ge 3$ is odd},\\
\{\lceil\frac{(n+1)}{2}\rceil\}, & \text {for $\min(mw(\K_{(1,n),\I}))$ if $n\ge 3$ is even},\\
\{1,n+1\}, & \text{for $\max(mw(\K_{(1,n),\I}))$}. 
\end{cases}
\end{equation*} 
\end{cor}
\begin{proof}
If $n$ is odd, $n+1$ is even and we have two central indexes namely, $\lceil\frac{(n+1)}{2}\rceil$ and $\lceil\frac{(n+1)}{2}\rceil + 1$ allowing minimal mirror image vertex index differences. Hence, $\min(mw(\K_{(1,n),\I}))= 2\sum\limits_{i=1}^{\lceil \frac{(n+1)}{2}\rceil-1}i+\lceil\frac{(n+1)}{2}\rceil= 2\sum\limits_{i=\lceil\frac{(n+1)}{2}\rceil -1}^{n+1}{|(\lceil \frac{(n+1)}{2}\rceil +1)-i|}+\lceil \frac{(n+1)}{2}\rceil$. 

If $n$ is even, $n+1$ is odd and hence we have a unique central vertex index allowing  minimal mirror image vertex index differences to be exactly, $\lceil \frac{(n+1)}{2}\rceil$. Hence, $\min(mw(\K_{(1,n),\I}))=2\sum\limits_{i=1}^{\lceil\frac{(n+1)}{2}\rceil-1}i)=2\sum \limits_{i= \lceil \frac{(n+1)}{2}\rceil-1}^{n+1} {|\lceil \frac{(n+1)}{2}\rceil - i|)}$. Hence we have the result as
\begin{equation*} 
\ell \in
\begin{cases}
\{\lceil \frac{(n+1)}{2}\rceil, \lceil \frac{(n+1)}{2}\rceil + 1\}, & \text{for $\min(mw(\K_{(1,n),\I}))$ if $n \ge 3$ and is odd},\\
\{\lceil\frac{(n+1)}{2}\rceil\}, & \text{for $\min(mw(\K_{(1,n),\I}))$ if $n \ge 3$ and is even},
\end{cases}
\end{equation*} 
follows.

Since $|1-(i+1)|$, for all $i \in \{2, 3,\ldots, n\}$, is  equal to $|(n+1)-i|$, for all $i\in \{1, 2,\ldots, n-1\}$ and assures maximal vertex index differences, the result $\ell \in \{1,n+1\}$ for determining $\max(mw(\K_{(1,n),\I}))$ follows.
\end{proof}

\subsection{Sprout  Complete Bi-partite Graphs}

A complete bi-partite graph $K_{(n,m)}$ has $V(K_{(n,m)}) = \{d_i: 1 \le i \le n\}\bigcup \{d'_i: 1 \le i \le m \}$ and $E(K_{(n,m)}) = \{d_id'_j: 1 \le i \le n$, $1 \le j \le m\}$.
\begin{prop}
For a complete bi-partite graph $K_{(n,m)}$, $n,m \ge 2$ we have\\
(i) $\min(mw(\K_{(n,m),{\I^{\ast}_1}}))=
\begin{cases}
\frac{nm}{2}(n+m-1)+\frac{m}{2}(m+1), & \text{if $n+m$ is odd},\\
\frac{nm}{2}(n+m), & \text{if $n+m$ is even}. 
\end{cases}$ \\\\
(ii) $\max(mw(\K_{(n,m),{\I^{\prime}_2}}))=
\begin{cases}
(n(n-1)+\frac{nm}{2}(n+m), & \text{if $n+m$ is even},\\
(n(\lfloor\frac{n+m}{2}\rfloor-1)+\frac{n}{2}(n+1))(n+m), & \text{if $n+m$ is odd}. 
\end{cases}$ 
\end{prop}
\begin{proof}
\textbf{(i)} Consider a complete bi-partite graph $K_{(n,m)}$, $n,m \ge 2$ and $n\ge m$. Without loss of generality let the left column have $n$ vertices and the right column have $m$ vertices. Label the vertices according to $\I^{\ast}_1$ as follows; the left column from top down, $v_1, v_{m+2}, v_{m+3}, \ldots , v_{m+n}$ and the right column from top down, $v_2, v_3, v_4, \ldots , v_{m+1}$. Clearly in terms of the MMAW-Principle, the maximum number of edges have been removed from $K_{(n+m)}$, all with maximum index difference, to construct $K_{(n,m)}$.

\ni \textbf{Subcase (i)(a):} Assume that $n+m$ is odd. The index differences $((m+2)+i)-j$, where $1\le i\le n-2$ and $2\le j\le m+1$, can be written in a $(n-1)\times (n+m)-2$ matrix form as 
\begin{equation*}
\A =
\begin{pmatrix}
1 & 2 & 3 & \ldots & m & 0 & 0 & \ldots & 0\\
0 & 2 & 3 & 4 & \ldots & m+1 & 0 & \ldots & 0\\
\hdotsfor{9}\\
0 & 0 & \ldots & 0 & n-1 & n & n+1 & \ldots & (n+m)-2
\end{pmatrix}.
\end{equation*}

\ni Hence, we have $\min(mw(\K_{(n,m),{\I^{\ast}_1}})) = \sum\limits^{n}_{i=1}\sum\limits^{m}_{j=1}a_{ij}$, $a_{ij} \in \A$. 

Alternatively, let $t=n+m$. It follows from $\A$ that the matrix can rather be written as a $m$ x $(n+m)-2$ triangular array  $\A^{\ast}$, with each row having odd number of entries namely

\begin{center}
$1$, $2$, $3$, $\ldots$, $m$, $m+1$, $\ldots$, $\frac{(n+m)-1}{2}$, $\ldots$, $n-2$, $n-1$, $\ldots$, $t-4$, $t-3$, $t-2$\\ 
$2$, $3$, $\ldots$, $m$, $m+1$, $\ldots$$,\frac{(n+m)-1}{2}$, $\ldots$, $n-2$, $n-1$$,\ldots$, $t-4$, $t-3$\\
$3$, $\ldots$, $m$, $m+1$, $\ldots$$,\frac{(n+m)-1}{2}$, $\ldots$, $n-2$, $n-1$, $\ldots$, $t-4$\\
$\vdots$\\ 
$m$, $m+1$, $\ldots$$,\frac{(n+m)-1}{2}$, $\ldots$, $n-2$, $n-1$.
\end{center}

\ni Clearly, $\min(mw(\K_{(n,m),{\I^{\ast}_1}}))>\sum\limits_{a_{ij} \in \A^{\ast}} a_{ij}$. Then, the above expressions can be written as $\min(mw(\K_{(n,m),{\I^{\ast}_1}}))>\sum\limits^{(n-1)}_{i=1}\sum\limits^{i+(m-1)}_{j=i}j$. Equality is obtained by adding the index difference $k-1$, where $2\le k \le m+1$. Hence, $\min(mw(\K_{(n,m),{\I^{\ast}_1}}))=\sum\limits^{(n-1)}_{i=1}\sum\limits^{i+(m-1)}_{j=i}j + \sum\limits^{m}_{i=1}i$. Therefore, the subcase (i)(a) is settled.

\ni \textbf{Subcase (i)(b):} Assume $n+m$ is even. Similar reasoning can be applied as in subcase (i)(a) except the fact that each row has even number of entries.

\ni \textbf{(ii)} Consider a complete bi-partite graph $K_{(n,m)}$, $n,m \ge 2$ and $n\le m$. Without loss of generality let the left column have $n$ vertices and the right column have $m$ vertices. Label the vertices according to $\I^{\prime}_2$ as follows; the left column from top down, $v_1, v_2, v_3, \ldots , v_n$ and the right column from top down, $v_{n+1}, v_{n+2}, v_{n+3}, \ldots , v_{n+m}$. Clearly, by the MmAW-Principle, the maximum number of edges have been removed from $K_{(n+m)}$, all with minimum index difference, to construct $K_{(n,m)}$.

\ni \textbf{Subcase (ii)(a):} Assume $n+m$ is even. The index differences $(n+i) - j$, $1 \le i \le m$ and $1 \le j \le n$ can be written in a $n$ x $m$ matrix form as 
\begin{equation*}
\A =
\begin{pmatrix}
1 & 2 & 3 & \ldots & m\\
2 & 3 & 4 & \ldots & m+1\\
\hdotsfor{5}\\
n & n+1 & n+2 & \ldots & (n+m)-1
\end{pmatrix}.
\end{equation*}
\ni Hence, we have $\max(mw(\K_{(n,m),{\I^{\prime}_2}}))=\sum\limits^{n}_{i=1}\sum\limits^{m}_{j=1}a_{ij}$, $a_{ij} \in \A$.

Alternatively, let $t=n+m$. It follows from $\A$ that the matrix can rather be written as a $(n+m)-1$ x $n$ triangular array  $\A^*$, with each row having odd number of entries namely

\begin{center}
$1$, $2$, $3$, $\ldots$, $n$, $n+1$, $\ldots$, $\frac{n+m}{2}$, $\ldots$, $m-1$, $m$, $\ldots$, $t-3$, $t-2$, $t-1$\\ 
$2$, $3$, $\ldots$, $n$, $n+1$, $\ldots$$,\frac{n+m}{2}$, $\ldots$, $m-1$, $m$$,\ldots$, $t-3$, $t-2$\\
$3$, $\ldots$, $n$, $n+1$, $\ldots$$,\frac{n+m}{2}$, $\ldots$, $m-1$, $m$, $\ldots$, $t-3$\\
$\vdots$\\
$n$, $n+1$, $\ldots$$,\frac{n+m}{2}$, $\ldots$, $m-1$, $m$.
\end{center}

Clearly, $\max(mw(\K_{(n,m),{\I^{\prime}_2}}))=\sum\limits_{a_{ij} \in \A^{\ast}} a_{ij}$. The above expressions  can be written as $\max(mw(\K_{(n,m),{\I^{\prime}_2}}))=(n(n-1) + \frac{1}{2}n(m-1))(n+m)+\frac{n}{2}(n+m)$. Hence, the subcase (ii)(a) is as also settled.

Subcase (ii)(b) Assume $n+m$ is odd. Similar reasoning as in subcase (ii)(a) except each row has even number of entries.
\end{proof}

\subsection{Sprouting of an Edge-joint Graph}

\ni Let us first recall the definition of the edge-joint graph of two given graphs. 

\begin{defn}{\rm 
\cite{KS4} The \textit{edge-joint} of two simple undirected graphs $G$ and $H$ is the graph obtained by adding the edge $vu_{\arrowvert_{_{v\in V(G),u \in V(H)}}}$, and is denoted by $G\rightsquigarrow_{vu}H$.}
\end{defn}

Consider the graphs $G$ and $H$ on $n$ and $m$ vertices respectively, with $m\le n$. Let the vertices of $G$ be labeled according to the index pattern $\I_1 = \{v_1, v_2, v_3, \ldots, v_n\}$ and the vertices of $H$ be labeled according to the index pattern $\I_2 = \{u_1, u_2, u_3, \ldots, u_m\}$. In the edge-joint graph $G\rightsquigarrow_{v_k u_l}H$, relabel the vertices of graph $H$ to $v_{n+1}, v_{n+2}, \ldots, v_{n+m}$. Also, let the new index pattern be $\I= \{v_1,v_2,v_3, \ldots, v_n, v_{n+1}, v_{n+2}, \ldots, v_{n+m}\}$. Invoking this concept, we have the next result.

\begin{thm}
For the graphs $G$ and $H$ on $n$ and $m$ vertices respectively, with $m \le n$, we have
\begin{enumerate}\itemsep0mm
\item[(i)] $mw((\G\rightsquigarrow_{v_ku_l}\H)_{\I})=mw(\G_{\I_1})+mw(\H_{\I^{\prime}_2})+|k-(l+n)|$.
\item[(ii)] $mw((\H\rightsquigarrow_{u_lv_k}\G)_{\I})=mw(\G_{\I^{\prime}_1})+mw(\H_{\I_2})+|(k+m)-l|$. 
\end{enumerate}
\end{thm}
\begin{proof}
\textbf{(i)} In graph $G$ indexing did not change and hence $mw(\G_{\I_1})$ remains the same. In graph $H$ the indexing changed consistently with $+n$ and hence for each pair of adjacent vertices say, $v_{i+n}, v_{j+n}$ we have $|(i+n)-(j+n)|=|i-j|$. Thus, $mw(\H_{\I^{\prime}_2})$ remains the same. Finally, the arc-weight of the new arc $(v_k,v_{l+n})=|k-(l+n)|$ is evident and hence the result follows.

\ni \textbf{(ii)} Similar reasoning as in (i).
\end{proof}


\section{Application to Certain Small Graphs}

\subsection{Sprout Wheels}

The next result follows from Proposition \ref{Prop-4.7} and Theorem \ref{Thm-4.9}. A wheel is defined as $W_{n+1} = C_n + K_1$.
 
\begin{prop}
For a wheel $W_{n+1}$, $n\ge 4$ we have

\begin{enumerate}\itemsep0mm
\item[(i)] $\min(mw(\W_{n+1,{\I^{\ast}_1}}))=\min(mw(\C_{n,{\I^{\ast}_1}}))+\min(mw(\K_{(1,n),{\I^{\ast}_1}}))+2$.
\item[(ii)] $\max(mw(\W_{n+1,{\I^{\prime}_2}})) = \max(mw(\C_{n,{\I^{\prime}_2}})) + \max(mw(\K_{(1,n),{\I^{\prime}_2}}))$.

\end{enumerate}\end{prop}
\begin{proof}
\textbf{(i)} Consider the index pattern of $K_{(1,n)}$ in accordance to Theorem \ref{Thm-4.9} and without loss of generality, let the central vertex be $v_t$, where $t=\lceil\frac{(n+1)}{2}\rceil$. Adding the cycle $C_n$ changes the index of only arcs. We now have arcs $(v_{t_1},v_{t_2})$ and either $t_1 =\lceil \frac{n}{2}\rceil-1$, $t_2 = \lceil \frac{n}{2}\rceil +1$ and $(v_1, v_{n+1})$ or $t_1 = \lceil \frac{(n+1)}{2}\rceil -1$, $t_2 = \lceil \frac{(n+1)}{2}\rceil +1$ and $(v_1, v_{n+1})$. If we consider the cycle $C_n$ only we have an increase in $\min(mw(\C_{n,{\I^{\ast}_1}} ))$ of either $(\lceil \frac{n}{2}\rceil +1)-(\lceil \frac{n}{2}\rceil-1)=2$ or $(\lceil \frac{(n+1)}{2}\rceil +1) - (\lceil \frac{(n+1)}{2}\rceil-1)=2$. Hence, in both cases, we have an increase of $1$ and therefore, a total increase of $2$ is effected and hence the first part of the result is settled.

\textbf{(ii)} Consider the cycle $C_n$ and label the vertices according to $\I^{\prime}_2$ as determined in Proposition \ref{Prop-4.7}. Without loss of generality add the central vertex $v_1$ (See Corollary \ref{Cor-4.10}) and add $1$ to each index of the cycle vertices, and denote this index pattern of $C_n$ to be $\I'_2$. Denote the index pattern of $K_{(1,n)}$ to be $\I''_2$. For every vertex $v_i$ in the cycle of $W_{n+1}$ we now have the pattern index $\I^{\prime}_2$ plus $1$. Considering the cycle only and invoking Lemma 1.4, we have $\max(mw(\C_{n,{\I'_2}})) = \max(mw(\C_{n,{\I^{\prime}_2}}))$.  Considering the star $K_{(1,n)}$ only we have $\max(mw(\K_{(1,n),{\I''_2}})) = \max(mw(K_{(1,n),{\I^{\prime}_2}}))$. Therefore part (ii) of the result also follows.
\end{proof}

\subsection{Sprout Ladder Graphs}

A ladder $L_n$, where $n\ge 3$, is defined to be $L_n = (P_n \cup P_n)+\{d_id'_i:2\le i\le n-1\}$. The edges $\{d_id'_i: 2 \le i \le n-1\}$ are called {\em steps}.

\begin{prop}
For a ladder $L_n$, we have $\min(mw(\L_{n,{\I^{\ast}_1}}))= 2\min(mw(\P_{n,{\I^{\ast}_1}}))+n(n-2)$.
\end{prop}
\begin{proof}
Let the ladder be constructed upright, that is, left and right pillars are both vertical $P_n$ with horizontal steps. Label the left pillar from top down, $v_1, v_2, \ldots , v_n$ and label the right pillar from top down, $v_{n+1}, v_{n+2}, v_{n+3}, \ldots , v_{2n}$. Clearly, by MMAW-Principle, the maximum number of edges have been removed from $K_{2n}$, all with maximum index difference, to construct $L_n$. For a step $v_iv_{i+n}$ we have index difference $n$. Since $n-2$ steps exist, the result, $\min(mw(\L_{n,{\I^{\ast}_1}}))=2\min(mw(\P_{n,{\I^{\ast}_1}}))+n(n-2)$ follows.
\end{proof}

Determining $\max(mw(\L_{n,{\I^{\prime}_2}}))$ is an open problem. 

\section{Conclusion}

The main focus of this study is that a matured sprout graph is a directed clone of an initial graph. In real world application, it means that the underlying graph of the resultant sprout graph must be known structurally or genetically in advance. If sprouts may re-direct through a probability function to arc elsewhere, the matured sprout graph may not resemble the initial graph. The latter calls for further research and could assist in understanding less predictable neurological growth of good or malicious networks or cell growth in biological structures.

The maxi-max arc-weight principle and the maxi-min arc-weight principle have been introduced. The authors suggest that these principles require further mathematical discussions for logical closure. Determining the minimum and maximum maturity weight of a wide rage of graphs classes and small graphs such as, the sun graph, the crown graph, the armed crown graph, the house graph, the rasta graph, the helm graph etc. might lead us to interesting results as well as methodologies of applications of these principles. 

Most of the results in this paper can be derived by simply labeling the vertices of a graph without the notion of sprouting. However, further in-depth research is required to explore the application of the process of sprouting as a dynamical concept. The graph structure can be conceptualised as a cancer type and $\sT_{\I}$ can be conceptualised as the aggressiveness index of cancerous growth (grade) whilst $t \in \sT_{\I}$ can be conceptualised as the stage or phase of growth.

Some of the open problems we have identified during our present study are the following. 
 
\begin{prob}{\rm
Prove or disprove the \textit{pattern conjecture} which states that if two different random index patterns of a graph $G$ say $\I_1$ and $\I_2$ result in $\sT_{\I_1}$ and $\sT_{\I_2}$ respectively, then $\sT_{\I_1} \subset \sT_{\I_2} \implies mw(\G_{\I_1})< mw(\G_{\I_2})$. }
\end{prob}

\begin{prob}{\rm
Determine $\min(mw(\G_{\I}))$ and $\max(mw(\G_{\I}))$ with $G\cong K_{(r_1, r_2, r_3, ..., r_n)}$, where $2 \le r_1 \le r_2 \le r_3 \le ... \le r_n, \forall r_i \in \N$,  being a complete $n$-partite graph.}
\end{prob}
 
\begin{prob}{\rm
Describe an algorithm to determine $max(mw(\T_{\I})), T$ a tree.}
\end{prob}
 
\begin{prob}{\rm
Describe a formal MMAW-Principle algorithm.}
\end{prob}

\end{document}